\documentclass[preprint,12pt]{elsarticle}
\usepackage{amssymb}
\usepackage{amsmath}
\usepackage{mathrsfs}
\usepackage{stmaryrd}
\usepackage{color}
\usepackage[all]{xy}
\usepackage{graphicx}
\numberwithin{equation}{section}
\newtheorem{theorem}{Theorem}[section]
\newtheorem{corollary}[theorem]{Corollary}

\newtheorem{lemma}[theorem]{Lemma}
\newtheorem{proposition}[theorem]{Proposition}
\newtheorem{remark}[theorem]{Remark}
\newtheorem{example}[theorem]{Example}

\newtheorem{definition}[theorem]{Definition}

\newproof{proof}{Proof}

\journal{Algebra Colloquium}

\begin{document}

\begin{frontmatter}

\title{\textbf{A new class of partial orders}}

\author[1]{Huihui Zhu\corref{cor}}
\ead{hhzhu@hfut.edu.cn}
\cortext[cor]{Corresponding author}

\author[1]{Liyun Wu}
\ead{wlymath@163.com}

\address[1]{School of Mathematics, Hefei University of Technology, Hefei 230009, China.}

\begin{abstract} Let $R$ be a unital $*$-ring. For any $a,w,b\in R$, we apply the defined $w$-core inverse to define a new class of partial orders in $R$, called the $w$-core partial order. Suppose $a,b\in R$ are $w$-core invertible. We say that $a$ is below $b$ under the $w$-core partial order, denoted by $a\overset{\tiny{\textcircled{\#}}}\leq_w b$, if $a_w^{\tiny{\textcircled{\#}}} a=a_w^{\tiny{\textcircled{\#}}} b$ and $awa_w^{\tiny{\textcircled{\#}}} =bwa_w^{\tiny{\textcircled{\#}}}$, where $a_w^{\tiny{\textcircled{\#}}}$ denotes the $w$-core inverse of $a$. Characterizations of the $w$-core partial order are given. Also, the relationships with several types of partial orders are considered. In particular, we show that the core partial order coincides with the $a$-core partial order, and the star partial order coincides with the $a^*$-core partial order.
\end{abstract}


\end{frontmatter}



\section{Introduction}

There are many types of partial orders based on generalized inverses in mathematical literature, such as the minus partial order \cite{Hartwig1980}, the plus partial order \cite{Hartwig1980}, the sharp partial order \cite{Mitra1987}, the star partial order \cite{Drazin1978}, the diamond partial order \cite{Lebtahi2014}, the core partial order \cite{Baksalary2010} and so on. These partial orders are investigated in different settings such as complex matrices and rings.

In this paper, we aim to introduce the $w$-core partial order based on our defined $w$-core inverses \cite{Zhu2020} and to give its several characterizations and properties. The paper is organized as follows. In Section 2, we define a class of partial order and establish its characterizations. In Section 3, the relationships between the $w$-core partial order and other partial orders are considered. It is proved that the $w$-core partial order is between the core partial order and the diamond partial order. We show in Theorem \ref{three class partial orders} that the star partial order and the core partial order are both instances of the $w$-core partial order. More precisely, the core partial order coincides with the $a$-core partial order, and the star partial order coincides with the $a^*$-core partial order. Then, the equivalence between $a\overset{\tiny{\textcircled{\#}}}\leq_w b$ and $b-a\overset{\tiny{\textcircled{\#}}}\leq_w b$ is derived, under certain conditions. In the end, the reverse order law for the $w$-core inverse is given.

Let us now recall several notions of generalized inverses in rings. Let $R$ be an associative ring with the identity 1. An element $a\in R$ is called (von Neumann) regular if there is some $x\in R$ such that $a=axa$. Such an $x$ is called an inner inverse of $a$, and is denoted by $a^{-}$.  If in addition, $xax=x$, then $a$ is called reflexive or $\{1,2\}$-invertible. Such an $x$ is called a reflexive inverse of $a$, and is denoted by $a^+$. Further, an element $a\in R$ is called group invertible if there exists a reflexive inverse $a^+$ of $a$ which commutes with $a$. Such an element $a^+$ is called a group inverse of $a$. It is unique if it exists, and is denoted by $a^\#$. We denote by $R^\#$ the set of all group invertible elements in $R$.

Given any $a,d\in R$, $a$ is invertible along $d$ \cite{Mary2011} if there exists some $b\in R$ such that $bad=d=dab$ and $b\in dR \cap Rd$. Such an element $b$ is called the inverse of $a$ along $d$. It is unique if it exists, and is denoted by $a^{\parallel d}$. As usual, by the symbol $R^{\parallel d}$ we denote the set of all invertible elements along $d$. It is known from \cite[Theorem 2.2]{Mary2013} that $a\in R^{\parallel d}$ if and only if $d\in dadR \cap Rdad$. Mary in \cite[Theorem 11]{Mary2011} proved that $a\in R^\#$ if and only if $a\in R^{\parallel a}$ if and only if $1\in R^{\parallel a}$. Moreover, $a^\#=a^{\parallel a}$ and $1^{\parallel a}=aa^\#$. More results on the inverse along an element can be referred to \cite{Mary2013,Zhu2016,Zhu20171}.

Throughout this paper, we assume that $R$ is a unital $*$-ring, that is a ring with unity 1 and an involution $*$ satisfying $(x^*)^*=x$, $(xy)^*=y^*x^*$ and $(x+y)^*=x^*+y^*$ for all $x,y\in R$. An element $a\in R$ (with involution) is Moore-Penrose invertible \cite{Penrose1955} if there is some $x\in R$ such that
\begin{center}
(i) $axa=a$, (ii) $xax=x$, (iii) $(ax)^*=ax$, (iv) $(xa)^*=xa$.
\end{center}
Such an $x$ is called a Moore-Penrose inverse of $a$. It is unique if it exists, and is denoted by $a^\dag$. We denote by $R^\dag$ the set of all Moore-Penrose invertible elements in $R$. It was proved in \cite{Mary2011,Zhu20171} that $a\in R^\dag$ if and only if $a^{\parallel a^*}$ exists if and only if $(a^*)^{\parallel a}$ exists. In this case, $a^\dag=a^{\parallel a^*}=((a^*)^{\parallel a})^*$.

If $a$ and $x$ satisfy the equations (i) $axa=a$ and  (iii) $(ax)^*=ax$, then $x$ is called a $\{1,3\}$-inverse of $a$, and is denoted by $a^{(1,3)}$. If $a$ and $x$ satisfy the equations (i) $axa=a$ and  (iv) $(xa)^*=xa$, then $x$ is called a $\{1,4\}$-inverse of $a$, and is denoted by $a^{(1,4)}$. We denote by $R^{(1,3)}$ and $R^{(1,4)}$ the sets of all $\{1,3\}$-invertible and $\{1,4\}$-invertible elements in $R$, respectively. It is well known that $a\in R^\dag$ if and only if $a\in R^{(1,3)} \cap R^{(1,4)}$ if and only if $a\in aa^*R \cap Ra^*a$ if and only if $a\in aa^*aR $ if and only if $a\in Raa^*a$. In this case, $a^\dag=a^{(1,4)}aa^{(1,3)}$.

The core inverse of complex matrices was firstly introduced by Baksalary and Trenkler  \cite{Baksalary2010}. In 2014, Raki\'{c} et al. \cite{Rakic2014} extended the core inverse of a complex matrix to an element in a unital $*$-ring. They showed that the core inverse of $a\in R$ is the solution of the following five equations
\begin{center}
 (1)~$axa=a$, (2)~$xax=x$, (3)~$ax^2=x$, (4)~$xa^2=a$, (5)~$(ax)^*=ax$.
\end{center}
The core inverse of $a\in R$ is unique if it exists, and is denoted by $a^{\tiny{\textcircled{\#}}}$. By $R^{\tiny{\textcircled{\#}}}$ we denote the set of all core invertible elements in $R$. In \cite[Theorem 2.6]{Xu2017}, Xu et al. showed that $a\in R^{\tiny{\textcircled{\#}}}$ if and only if $a\in R^\# \cap R^{(1,3)}$. In this case, the expression of the core inverse can be given as $a^{\tiny{\textcircled{\#}}}=a^\#aa^{(1,3)}$.

Recently, the present authors \cite{Zhu2020} defined the $w$-core inverse in a ring $R$. Given any $a,w\in R$, we say that $a$ is $w$-core invertible if there exists some $x\in R$ such that $awx^2=x$, $xawa=a$ and $(awx)^*=awx$. Such an $x$ is called a $w$-core inverse of $a$. It is unique if it exists, and is denoted by $a_w^{\tiny{\textcircled{\#}}}$.  Also, the $w$-core inverse $x$ of $a$ satisfies $awxa=a$ and $xawx=x$ (see \cite[Lemma 2.2]{Zhu2020}). By $R_w^{\tiny{\textcircled{\#}}}$ we denote the set of all $w$-core invertible elements in $R$. It was proved in \cite[Theorem 2.11]{Zhu2020} that $a \in R_w^{\tiny{\textcircled{\#}}}$ if and only if $w\in R^{\parallel a}$ and $a\in R^{(1,3)}$. Moreover, $a_w^{\tiny{\textcircled{\#}}}=w^{\parallel a}a^{(1,3)}$.

Several well known partial orders on a ring $R$ are given below.

(1) The minus partial order: $a\overset{-}\leq b$  if and only if there exists an inner inverse $a^-\in R$ of $a$ such that $a^-a =a^-b$ and $aa^- = ba^-$.

(2) The plus partial order:  $a\overset{+}\leq b$ if and only if  there exists a reflexive inverse $a^+\in R$ of $a$ such that $a^+a =a^+b$ and $aa^+ = ba^+$.

(3) The sharp partial order: $a\overset{\#}\leq b$  if and only if there exists the group inverse $a^\#\in R$ of $a$ such that $a^\#a=a^\#b$ and $aa^\#=ba^\#$.

(4) The star partial order:  $a\overset{*}\leq b$  if and only if $a^*a=a^*b$ and $aa^*=ba^*$. In particular, if $a\in R^\dag$, then $a\overset{*}\leq b$  if and only if $a^\dag a=a^\dag b$ and $aa^\dag=ba^\dag$.

(5) The diamond partial order: $a\overset{\diamond}\leq b$ if and only if $aa^*a=ab^*a$, $aR\subseteq bR$ and $Ra\subseteq Rb$.

(6) The core partial order: $a\overset{\tiny{\textcircled{\tiny\#}}}\leq b$ if and only if there exists the core inverse $a^{\tiny{\textcircled{\#}}}\in R$ of $a$ such that $a^{\tiny{\textcircled{\#}}}a=a^{\tiny{\textcircled{\#}}}b$ and $aa^{\tiny{\textcircled{\#}}}=ba^{\tiny{\textcircled{\#}}}$.

\section{The $w$-core partial order}

In this section, we aim to define a class of partial orders and to give its characterizations in a ring $R$.

\begin{definition} \label{w core relation} Let $a,b,w\in R$ with $a,b\in R_w^{\tiny{\textcircled{\#}}} $. We say that $a$ is below $b$ under the $w$-core relation and write $a\overset{\tiny{\textcircled{\#}}}\leq_w b$ if $a_w^{\tiny{\textcircled{\#}}} a=a_w^{\tiny{\textcircled{\#}}} b$ and $awa_w^{\tiny{\textcircled{\#}}} =bwa_w^{\tiny{\textcircled{\#}}}$.
\end{definition}

We next show that the relation $a\overset{\tiny{\textcircled{\#}}}\leq_w b$ is a partial order. First, an auxiliary lemma is given below.

\begin{lemma} \label{w-core par Lemma} Let $a,b,w\in R$ with $a,b\in R_w^{\tiny{\textcircled{\#}}} $. If $a_w^{\tiny{\textcircled{\#}}} a=a_w^{\tiny{\textcircled{\#}}} b$ and $awa_w^{\tiny{\textcircled{\#}}} =bwa_w^{\tiny{\textcircled{\#}}}$, then we have

\emph{(i)} $a_w^{\tiny{\textcircled{\#}}}a=b_w^{\tiny{\textcircled{\#}}}a$.

\emph{(ii)} $awa_w^{\tiny{\textcircled{\#}}}=awb_w^{\tiny{\textcircled{\#}}}$.

\emph{(iii)} $awb_w^{\tiny{\textcircled{\#}}}a=a$.

\emph{(iv)} $b_w^{\tiny{\textcircled{\#}}}a_w^{\tiny{\textcircled{\#}}}=(a_w^{\tiny{\textcircled{\#}}})^2$.
\end{lemma}

\begin{proof} (i) As $awa_w^{\tiny{\textcircled{\#}}}a=a$ and $aw(a_w^{\tiny{\textcircled{\#}}})^2=a_w^{\tiny{\textcircled{\#}}}$, then we have
\begin{eqnarray*}
b_w^{\tiny{\textcircled{\#}}}a&=&b_w^{\tiny{\textcircled{\#}}}awa_w^{\tiny{\textcircled{\#}}}a
=b_w^{\tiny{\textcircled{\#}}}bwa_w^{\tiny{\textcircled{\#}}}a
=b_w^{\tiny{\textcircled{\#}}}bwaw(a_w^{\tiny{\textcircled{\#}}})^2a\\
&=&b_w^{\tiny{\textcircled{\#}}}bwbw(a_w^{\tiny{\textcircled{\#}}})^2a
=bw(a_w^{\tiny{\textcircled{\#}}})^2a=aw(a_w^{\tiny{\textcircled{\#}}})^2a\\
&=&a_w^{\tiny{\textcircled{\#}}}a.
\end{eqnarray*}

(ii) We have $awa_w^{\tiny{\textcircled{\#}}}=awb_w^{\tiny{\textcircled{\#}}}$. Indeed,
\begin{eqnarray*}
awb_w^{\tiny{\textcircled{\#}}}&=&(awa_w^{\tiny{\textcircled{\#}}}a)wb_w^{\tiny{\textcircled{\#}}}
=aw(a_w^{\tiny{\textcircled{\#}}}b)wb_w^{\tiny{\textcircled{\#}}}
=(awa_w^{\tiny{\textcircled{\#}}})^*(bwb_w^{\tiny{\textcircled{\#}}})^*\\
&=&(bwb_w^{\tiny{\textcircled{\#}}}awa_w^{\tiny{\textcircled{\#}}})^*
=(bwb_w^{\tiny{\textcircled{\#}}}bwa_w^{\tiny{\textcircled{\#}}})^*=(bwa_w^{\tiny{\textcircled{\#}}})^*
=(awa_w^{\tiny{\textcircled{\#}}})^*\\
&=&awa_w^{\tiny{\textcircled{\#}}}.
\end{eqnarray*}

(iii) By (ii), $awb_w^{\tiny{\textcircled{\#}}}a=awa_w^{\tiny{\textcircled{\#}}}a=a$.

(iv) Note that $a_w^{\tiny{\textcircled{\#}}}a=b_w^{\tiny{\textcircled{\#}}}a$ in (i). Then $b_w^{\tiny{\textcircled{\#}}}a_w^{\tiny{\textcircled{\#}}}
=b_w^{\tiny{\textcircled{\#}}}aw(a_w^{\tiny{\textcircled{\#}}})^2
=a_w^{\tiny{\textcircled{\#}}}aw(a_w^{\tiny{\textcircled{\#}}})^2
=(a_w^{\tiny{\textcircled{\#}}})^2$.
 \hfill$\Box$
\end{proof}

\begin{theorem}
The relation $a\overset{\tiny{\textcircled{\#}}}\leq_w b$ of Definition $\ref{w core relation}$ is a partial order on $R$.
\end{theorem}

\begin{proof} To prove that $a\overset{\tiny{\textcircled{\#}}}\leq_w b$ is a partial order. It suffices to show (1) reflexivity, i.e., $a\overset{\tiny{\textcircled{\#}}}\leq_w a$, (2) antisymmetry, i.e., $a\overset{\tiny{\textcircled{\#}}}\leq_w b$  and $b\overset{\tiny{\textcircled{\#}}}\leq_w a$  imply $a=b$, (3) transitivity, i.e., $a\overset{\tiny{\textcircled{\#}}}\leq_w b$ and $b\overset{\tiny{\textcircled{\#}}}\leq_w c$  give $a\overset{\tiny{\textcircled{\#}}}\leq_w c$.

(1) The reflexivity is clear.

(2) Suppose $a\overset{\tiny{\textcircled{\#}}}\leq_w b$, i.e., $a_w^{\tiny{\textcircled{\#}}} a=a_w^{\tiny{\textcircled{\#}}} b$ and $awa_w^{\tiny{\textcircled{\#}}} =bwa_w^{\tiny{\textcircled{\#}}}$. Then $a=awa_w^{\tiny{\textcircled{\#}}}a=bwa_w^{\tiny{\textcircled{\#}}}a=bwb_w^{\tiny{\textcircled{\#}}}a$ by Lemma \ref{w-core par Lemma}(i). Suppose in addition that $b\overset{\tiny{\textcircled{\#}}}\leq_w a$, i.e., $b_w^{\tiny{\textcircled{\#}}} b=b_w^{\tiny{\textcircled{\#}}} a$ and $bwb_w^{\tiny{\textcircled{\#}}} =awb_w^{\tiny{\textcircled{\#}}}$. Then $b=bwb_w^{\tiny{\textcircled{\#}}}b=bwb_w^{\tiny{\textcircled{\#}}}a$, which together with $a=bwb_w^{\tiny{\textcircled{\#}}}a$ give $a=b$.

(3) Assume $a\overset{\tiny{\textcircled{\#}}}\leq_w b$ and $b \leq_w^{\tiny{\textcircled{\#}}} c$, i.e., $a_w^{\tiny{\textcircled{\#}}} a=a_w^{\tiny{\textcircled{\#}}} b$, $awa_w^{\tiny{\textcircled{\#}}} =bwa_w^{\tiny{\textcircled{\#}}}$, $b_w^{\tiny{\textcircled{\#}}} b=b_w^{\tiny{\textcircled{\#}}} c$ and $bwb_w^{\tiny{\textcircled{\#}}} =cwb_w^{\tiny{\textcircled{\#}}}$. Then, by the equality $awa_w^{\tiny{\textcircled{\#}}}=awb_w^{\tiny{\textcircled{\#}}}$ of Lemma \ref{w-core par Lemma}(ii), we get $a_w^{\tiny{\textcircled{\#}}}a=a_w^{\tiny{\textcircled{\#}}}b=a_w^{\tiny{\textcircled{\#}}}bwb_w^{\tiny{\textcircled{\#}}}b
=a_w^{\tiny{\textcircled{\#}}}bwb_w^{\tiny{\textcircled{\#}}}c=a_w^{\tiny{\textcircled{\#}}}awb_w^{\tiny{\textcircled{\#}}}c
=a_w^{\tiny{\textcircled{\#}}}awa_w^{\tiny{\textcircled{\#}}}c=a_w^{\tiny{\textcircled{\#}}}c$.

Similarly, we have
\begin{eqnarray*}
awa_w^{\tiny{\textcircled{\#}}}&=&bwa_w^{\tiny{\textcircled{\#}}}=bwb_w^{\tiny{\textcircled{\#}}}bwa_w^{\tiny{\textcircled{\#}}}
=cwb_w^{\tiny{\textcircled{\#}}}bwa_w^{\tiny{\textcircled{\#}}}=cwb_w^{\tiny{\textcircled{\#}}}awa_w^{\tiny{\textcircled{\#}}}
=cwa_w^{\tiny{\textcircled{\#}}}awa_w^{\tiny{\textcircled{\#}}}\\
&=&cwa_w^{\tiny{\textcircled{\#}}}.
\end{eqnarray*}

The proof is completed.
 \hfill$\Box$
\end{proof}

From now on, the partial order  $\overset{\tiny{\textcircled{\#}}}\leq_w $ is called the $w$-core partial order. The $w$-core partial order can be seen as an extension of the core partial order \cite{Baksalary2010}. However, the $w$-core partial order may not imply the core partial order in general. See Example \ref{ex1} below. Specially, by fixing the element $w\in R$, we in Theorem \ref{three class partial orders} below show that the $w$-core partial order $\overset{\tiny{\textcircled{\#}}}\leq_w$ coincides with the classical star partial order $\overset{*}\leq $ and the core partial order $\overset{\tiny{\textcircled{\#}}}\leq $, respectively.

\begin{example} \label{ex1} {\rm Let $R=M_2(\mathbb{C})$ be the ring of all $2 \times 2$ complex matrices and let the involution $*$ be the conjugate transpose. Take, for example, $a=\begin{bmatrix}
1 & 1 \\
0 & 0 \\
\end{bmatrix}$, $b=\begin{bmatrix}
1 & 1 \\
2 & -2 \\
\end{bmatrix}
$, $w=\begin{bmatrix}
           1 & 0 \\
           1 & 0 \\
         \end{bmatrix}
\in R$, then $a_w^{\tiny{\textcircled{\#}}}=
\begin{bmatrix}
\frac{1}{2} & 0 \\
0 & 0 \\
\end{bmatrix}$ and $a^{\tiny{\textcircled{\#}}}=
\begin{bmatrix}
           1 & 0 \\
           0 & 0 \\
         \end{bmatrix}$. We have $a_w^{\tiny{\textcircled{\#}}}a=a_w^{\tiny{\textcircled{\#}}}b=
\begin{bmatrix}
\frac{1}{2} & \frac{1}{2} \\
0 & 0 \\
\end{bmatrix}$ and $awa_w^{\tiny{\textcircled{\#}}}=bwa_w^{\tiny{\textcircled{\#}}}=
\begin{bmatrix}
1 & 0 \\
 0 & 0 \\
 \end{bmatrix}$. Hence, $a\overset{\tiny{\textcircled{\#}}}\leq_w b$. However, $aa^{\tiny{\textcircled{\#}}}=
\begin{bmatrix}
1 & 0 \\
 0 & 0 \\
 \end{bmatrix} \neq
 \begin{bmatrix}
1 & 0 \\
 2 & 0 \\
 \end{bmatrix}=ba^{\tiny{\textcircled{\#}}}$.}
\end{example}

We next give the characterization of the $w$-core partial order $a\overset{\tiny{\textcircled{\#}}}\leq_w b$ in $R$.

\begin{proposition} Let $a,b,w\in R$ with $a,b\in R_w^{\tiny{\textcircled{\#}}}$. Then the following conditions are equivalent{\rm:}

\emph{(i)} $a\overset{\tiny{\textcircled{\#}}}\leq_w b$.

\emph{(ii)} $a_w^{\tiny{\textcircled{\#}}}b=b_w^{\tiny{\textcircled{\#}}}a$, $bwa_w^{\tiny{\textcircled{\#}}}=awb_w^{\tiny{\textcircled{\#}}}$ and $awb_w^{\tiny{\textcircled{\#}}}a=a$.
\end{proposition}

\begin{proof} (i) $\Rightarrow$ (ii) It follows from Lemma \ref{w-core par Lemma}.

(ii) $\Rightarrow$ (i) We have $a_w^{\tiny{\textcircled{\#}}}a=a_w^{\tiny{\textcircled{\#}}}awb_w^{\tiny{\textcircled{\#}}}a
=a_w^{\tiny{\textcircled{\#}}}awa_w^{\tiny{\textcircled{\#}}}b=a_w^{\tiny{\textcircled{\#}}}b$, and $awa_w^{\tiny{\textcircled{\#}}}=(awb_w^{\tiny{\textcircled{\#}}}a)wa_w^{\tiny{\textcircled{\#}}}
=(bwa_w^{\tiny{\textcircled{\#}}})awa_w^{\tiny{\textcircled{\#}}}
=bw(a_w^{\tiny{\textcircled{\#}}}awa_w^{\tiny{\textcircled{\#}}})=bwa_w^{\tiny{\textcircled{\#}}}$, as required.
\hfill$\Box$
\end{proof}

\begin{proposition} \label{new add pro 2.6} Let $a,b,w\in R$ with $a,b\in R_w^{\tiny{\textcircled{\#}}}$. If $a\overset{\tiny{\textcircled{\#}}}\leq_w b$, then

\emph{(i)} $a_w^{\tiny{\textcircled{\#}}}bwb_w^{\tiny{\textcircled{\#}}}=b_w^{\tiny{\textcircled{\#}}}bwa_w^{\tiny{\textcircled{\#}}}
=a_w^{\tiny{\textcircled{\#}}}bwa_w^{\tiny{\textcircled{\#}}}=a_w^{\tiny{\textcircled{\#}}}$.

\emph{(ii)} $a_w^{\tiny{\textcircled{\#}}}awb_w^{\tiny{\textcircled{\#}}}=b_w^{\tiny{\textcircled{\#}}}awa_w^{\tiny{\textcircled{\#}}}=
b_w^{\tiny{\textcircled{\#}}}awb_w^{\tiny{\textcircled{\#}}}=a_w^{\tiny{\textcircled{\#}}}$.
\end{proposition}

\begin{proof}
(i) Given $a\overset{\tiny{\textcircled{\#}}}\leq_w b$, then $a_w^{\tiny{\textcircled{\#}}}b=a_w^{\tiny{\textcircled{\#}}} a=b_w^{\tiny{\textcircled{\#}}} a$ and $bwa_w^{\tiny{\textcircled{\#}}}=awa_w^{\tiny{\textcircled{\#}}}=awb_w^{\tiny{\textcircled{\#}}}$ by Lemma \ref{w-core par Lemma}. So, $a_w^{\tiny{\textcircled{\#}}}bwa_w^{\tiny{\textcircled{\#}}}=a_w^{\tiny{\textcircled{\#}}}awb_w^{\tiny{\textcircled{\#}}}
=a_w^{\tiny{\textcircled{\#}}}awa_w^{\tiny{\textcircled{\#}}}=a_w^{\tiny{\textcircled{\#}}}=b_w^{\tiny{\textcircled{\#}}}awa_w^{\tiny{\textcircled{\#}}}
=b_w^{\tiny{\textcircled{\#}}}bwa_w^{\tiny{\textcircled{\#}}}=b_w^{\tiny{\textcircled{\#}}}awb_w^{\tiny{\textcircled{\#}}}=
a_w^{\tiny{\textcircled{\#}}}bwb_w^{\tiny{\textcircled{\#}}}$.

(ii) By (i) and Lemma \ref{w-core par Lemma}.
\hfill$\Box$
\end{proof}

An element $e\in R$ is idempotent if $e=e^2$. If in addition, $e=e^*$, then $e$ is called a projection. It follows from \cite{Zhu2020} that if $a\in R_w^{\tiny{\textcircled{\#}}}$ then $a_w^{\tiny{\textcircled{\#}}}=w^{\parallel a}a^{(1,3)}$. So, $a_w^{\tiny{\textcircled{\#}}}aw=w^{\parallel a}a^{(1,3)}aw=w^{\parallel a}w$, $wa_w^{\tiny{\textcircled{\#}}}a=ww^{\parallel a}a^{(1,3)}a=ww^{\parallel a}$ and $awa_w^{\tiny{\textcircled{\#}}}=aww^{\parallel a}a^{(1,3)}=aa^{(1,3)}$. Clearly, $a_w^{\tiny{\textcircled{\#}}}aw$ and $wa_w^{\tiny{\textcircled{\#}}}a$ are both idempotents, and $awa_w^{\tiny{\textcircled{\#}}}$ is a projection.

In \cite[Lemma 2.1]{Marovt2016}, Marovt derived several characterizations for the idempotent $aa^\#$ in a ring. Inspired by Marovt's result, we establish several characterizations for the projection $awa_w^{\tiny{\textcircled{\#}}}$, the idempotents $a_w^{\tiny{\textcircled{\#}}}aw$ and $wa_w^{\tiny{\textcircled{\#}}}a$, respectively.

Given any $a\in R$, the symbol $a^{0}=\{x\in R:ax=0\}$ denotes all right annihilators of $a$. Dually, $^{0}a=\{x\in R:xa=0\}$ denotes all left annihilators of $a$. It should be noted that (see, e.g., \cite[Lemmas 2.5 and 2.6]{Rakic2014}) $aR=bR$ implies ${^0}a={^0}b$, and $Ra=Rb$ implies $a^0=b^0$ for any $a,b\in R$.

\begin{lemma} \label{projection} Let $a,w\in R$ with $a\in R_w^{\tiny{\textcircled{\#}}}$. Then the following conditions are equivalent{\rm :}

\emph{(i)} $p=awa_w^{\tiny{\textcircled{\#}}}$.

\emph{(ii)} $aR=pR$ for some projection $p\in R$.

\emph{(iii)} ${^0}a={^0}p$ for some projection $p\in R$.

\emph{(iv)} $a=pa$, ${^0}a\subseteq{^0}p$ for some projection $p\in R$.

\end{lemma}

\begin{proof}
(i) $\Rightarrow$ (ii) Given $p=awa_w^{\tiny{\textcircled{\#}}}$, then $pR=awa_w^{\tiny{\textcircled{\#}}}R\subseteq aR$. Also, $aR=(awa_w^{\tiny{\textcircled{\#}}}a)R=(pa)R\subseteq pR$.

(ii) $\Rightarrow$ (iii) is a tautology.

(iii) $\Rightarrow$ (iv) As $(1-p)p=0$ and ${^0}p={^0}a$, then $(1-p)a=0$ and hence $a=pa$.

(iv) $\Rightarrow$ (i) Note that $(1-awa_w^{\tiny{\textcircled{\#}}})a=0$. Then we get $(1-awa_w^{\tiny{\textcircled{\#}}})p=0$ and $p=awa_w^{\tiny{\textcircled{\#}}}p=(awa_w^{\tiny{\textcircled{\#}}}p)^*=pawa_w^{\tiny{\textcircled{\#}}}=awa_w^{\tiny{\textcircled{\#}}}$.
\hfill$\Box$
\end{proof}

\begin{theorem} \label{projection1} Let $a,b,w\in R$ with $a\in R_w^{\tiny{\textcircled{\#}}}$. Then the following conditions are equivalent{\rm :}

\emph{(i)} $a_w^{\tiny{\textcircled{\#}}}a=a_w^{\tiny{\textcircled{\#}}}b$.

\emph{(ii)} $w^{\parallel a}=w^{\parallel a}a^{(1,3)}b$.

\emph{(iii)} $a^*a=a^*b$.

\emph{(iv)} $a=awa_w^{\tiny{\textcircled{\#}}}b$.

\emph{(v)} There exists a projection $p\in R$ such that $aR=pR$ and $pa=pb$.

\emph{(vi)} There exists a projection $p\in R$ such that ${^0}p={^0}a$ and $pa=pb$.
\end{theorem}

\begin{proof}

(i) $\Rightarrow$ (ii) Note that $a_w^{\tiny{\textcircled{\#}}}=w^{\parallel a}a^{(1,3)}$ and $w^{\parallel a}\in Ra$. Then, there exists some $x\in R$ such that $w^{\parallel a}=xa=xaa^{(1,3)}a=w^{\parallel a}a^{(1,3)}a=a_w^{\tiny{\textcircled{\#}}}a=a_w^{\tiny{\textcircled{\#}}}b=w^{\parallel a}a^{(1,3)}b$.

(ii) $\Rightarrow$ (iii) Given $w^{\parallel a}=w^{\parallel a}a^{(1,3)}b$, then $a^*a=a^*(aww^{\parallel a})=a^*aw(w^{\parallel a}a^{(1,3)}b)=a^*(aww^{\parallel a})a^{(1,3)}b=a^*aa^{(1,3)}b=(aa^{(1,3)}a)^*b=a^*b$.

(iii) $\Rightarrow$ (iv) Since $a=(awa_w^{\tiny{\textcircled{\#}}})^*a=(wa_w^{\tiny{\textcircled{\#}}})^*a^*a$, we have $a=(wa_w^{\tiny{\textcircled{\#}}})^*a^*b=(awa_w^{\tiny{\textcircled{\#}}})^*b=awa_w^{\tiny{\textcircled{\#}}}b$.

(iv) $\Rightarrow$ (v) Write $p=awa_w^{\tiny{\textcircled{\#}}}$, then $p^2=p=p^*$ and $pR=awa_w^{\tiny{\textcircled{\#}}}R\subseteq aR=awa_w^{\tiny{\textcircled{\#}}}aR\subseteq pR$. Also, $pa=a=awa_w^{\tiny{\textcircled{\#}}}b=pb$.

(v) $\Rightarrow$ (vi) is a tautology.

(vi) $\Rightarrow$ (i) It follows from Lemma \ref{projection} that $p=awa_w^{\tiny{\textcircled{\#}}}$. Then $a=pa=pb=awa_w^{\tiny{\textcircled{\#}}}b$ and consequently $a_w^{\tiny{\textcircled{\#}}}a=a_w^{\tiny{\textcircled{\#}}}awa_w^{\tiny{\textcircled{\#}}}b=a_w^{\tiny{\textcircled{\#}}}b$.
\hfill$\Box$
\end{proof}

Given any $a,w\in R$ and $a\in R_w^{\tiny{\textcircled{\#}}}$, we next present several characterizations for $a_w^{\tiny{\textcircled{\#}}}aw$ and $wa_w^{\tiny{\textcircled{\#}}}a$ in a ring, respectively.

\begin{lemma}  \label{idempotent} Let $a,w,e\in R$ with $a\in R_w^{\tiny{\textcircled{\#}}}$. Then the following conditions are equivalent{\rm :}

\emph{(i)} $e=a_w^{\tiny{\textcircled{\#}}}aw$.

\emph{(ii)} $Re=Raw$ and $ea=a$.

\emph{(iii)} $e^0=(aw)^0$ and $ea=a$.

\emph{(iv)} $(aw)^0\subseteq e^0$ and $ea=a$.
\end{lemma}

\begin{proof} (i) $\Rightarrow$ (ii) Since $e=a_w^{\tiny{\textcircled{\#}}}aw$, we have $ea=a_w^{\tiny{\textcircled{\#}}}awa=a$ and $Re\subseteq Raw=Rawa_w^{\tiny{\textcircled{\#}}}aw=Rawe \subseteq Re$.

(ii) $\Rightarrow$ (iii) and (iii) $\Rightarrow$ (iv) are clear.

(iv) $\Rightarrow$ (i) Note that $ea=a$ concludes $ea_w^{\tiny{\textcircled{\#}}}=a_w^{\tiny{\textcircled{\#}}}$ by $aw(a_w^{\tiny{\textcircled{\#}}})^2=a_w^{\tiny{\textcircled{\#}}}$. Also, from $aw(1-a_w^{\tiny{\textcircled{\#}}}aw)=0$ and $(aw)^0\subseteq e^0$, it follows that $e=ea_w^{\tiny{\textcircled{\#}}}aw=a_w^{\tiny{\textcircled{\#}}}aw$.
\hfill$\Box$
\end{proof}

A similar characterization for the idempotent $wa_w^{\tiny{\textcircled{\#}}}a$ can also be obtained, whose proof is omitted.

\begin{lemma}  \label{idempotent0} Let $a,w,f\in R$ with $a\in R_w^{\tiny{\textcircled{\#}}}$. Then the following conditions are equivalent{\rm :}

\emph{(i)} $f=wa_w^{\tiny{\textcircled{\#}}}a$.

\emph{(ii)} $Ra=Rf$ and $fwa=wa$.

\emph{(iii)} $a^0=f^0$ and $fwa=wa$.

\emph{(iv)} $a^0\subseteq f^0$ and $fwa=wa$.
\end{lemma}

The characterization for $awa_w^{\tiny{\textcircled{\#}}}=bwa_w^{\tiny{\textcircled{\#}}}$ can also be derived similarly. First, a preliminary lemma is given.

\begin{lemma} \label{group result} {\rm \cite[Theorem 2.1]{Mary2013}} Let $a,w\in R$. Then the following conditions are equivalent{\rm:}

\emph{(i)} $w\in R^{\parallel a}$.

\emph{(ii)} $a\in awR$ and $aw\in R^\#$.

\emph{(iii)} $a\in Rwa$ and $wa\in R^\#$.

In this case, $w^{\parallel a}=a(wa)^\#=(aw)^\# a$.
\end{lemma}

\begin{theorem} \label{idempotent1} Let $a,b,w\in R$ with $a\in R_w^{\tiny{\textcircled{\#}}}$. Then the following conditions are equivalent{\rm :}

\emph{(i)} $awa_w^{\tiny{\textcircled{\#}}}=bwa_w^{\tiny{\textcircled{\#}}}$.

\emph{(ii)} $a=bww^{\parallel a}$.

\emph{(iii)} $awa=bwa$.

\emph{(iv)} $a(wa)^\#=b(wa)^\#$.

\emph{(v)} $a=bwa_w^{\tiny{\textcircled{\#}}}a$.

\emph{(vi)} There exists some $e\in R$ such that $Re=Raw$, $ea=a$ and $awe=bwe$.

\emph{(vii)} There exists some $e\in R$ such that $e^0=(aw)^0$, $ea=a$ and $awe=bwe$.

\emph{(viii)} There exists some $e\in R$ such that $(aw)^0 \subseteq e^0$, $ea=a$ and $awe=bwe$.

\emph{(ix)} There exists some $f\in R$ such that $Ra=Rf$, $af=bf$ and $fwa=wa$.

\emph{(x)} There exists some $f\in R$ such that $a^0=f^0$, $af=bf$ and $fwa=wa$.

\emph{(xi)} There exists some $f\in R$ such that $a^0\subseteq f^0$, $af=bf$ and $fwa=wa$.
\end{theorem}

\begin{proof}
(i) $\Rightarrow$ (ii) Given $awa_w^{\tiny{\textcircled{\#}}}=bwa_w^{\tiny{\textcircled{\#}}}$, then $a=awa_w^{\tiny{\textcircled{\#}}}a=bwa_w^{\tiny{\textcircled{\#}}}a=bww^{\parallel a}a^{(1,3)}a=bww^{\parallel a}$.

(ii) $\Rightarrow$ (iii) As $a=bww^{\parallel a}$, then $awa=(bww^{\parallel a})wa=bw(w^{\parallel a}wa)=bwa$.

(iii) $\Rightarrow$ (iv) Note that $a\in R_w^{\tiny{\textcircled{\#}}}$ implies $wa\in R^\#$ by Lemma \ref{group result} (i) $\Rightarrow$ (iii). Post-multiplying $awa=bwa$ by $((wa)^\#)^2$ gives $a(wa)^\#=b(wa)^\#$.

(iv) $\Rightarrow$ (v) As $a\in R_w^{\tiny{\textcircled{\#}}}$, then $w^{\parallel a}=a(wa)^\#$ in terms of Lemma \ref{group result}. Note that $w^{\parallel a}\in aR \cap Ra$. Then $a_w^{\tiny{\textcircled{\#}}}a=w^{\parallel a}a^{(1,3)}a=w^{\parallel a}$. Thus, $a=w^{\parallel a}wa=a(wa)^\#wa=b(wa)^\#wa=bwa(wa)^\#=bww^{\parallel a}=bwa_w^{\tiny{\textcircled{\#}}}a$.

(v) $\Rightarrow$ (vi) Write $e=a_w^{\tiny{\textcircled{\#}}}aw$, then $Re\subseteq Raw=R(awa_w^{\tiny{\textcircled{\#}}}a)w=Raw(a_w^{\tiny{\textcircled{\#}}}aw)=Rawe\subseteq Re$ and $ea=a_w^{\tiny{\textcircled{\#}}}awa=a$. Also, $awe=awa_w^{\tiny{\textcircled{\#}}}aw=aw=bwa_w^{\tiny{\textcircled{\#}}}aw=bwe$.

(vi) $\Rightarrow$ (vii) and (vii) $\Rightarrow$ (viii) are obvious.

(viii) $\Rightarrow$ (i) It follow from Lemma \ref{idempotent} that $e=a_w^{\tiny{\textcircled{\#}}}aw$ and $awa_w^{\tiny{\textcircled{\#}}}=awa_w^{\tiny{\textcircled{\#}}}awa_w^{\tiny{\textcircled{\#}}}
=awea_w^{\tiny{\textcircled{\#}}}=bwea_w^{\tiny{\textcircled{\#}}}=bwa_w^{\tiny{\textcircled{\#}}}$.

(i) $\Rightarrow$ (ix) Write $f=wa_w^{\tiny{\textcircled{\#}}}a$, then $Rf\subseteq Ra=Raf\subseteq Rf$, $af=awa_w^{\tiny{\textcircled{\#}}}a=bwa_w^{\tiny{\textcircled{\#}}}a=bf$ and $fwa=(wa_w^{\tiny{\textcircled{\#}}}a)wa=w(a_w^{\tiny{\textcircled{\#}}}awa)=wa$.

(ix) $\Rightarrow$ (x) and (x) $\Rightarrow$ (xi) are clear.

(xi) $\Rightarrow$ (i) Given $a^0\subseteq f^0$ and $fwa=wa$, then, by Lemma \ref{idempotent0}, $f=wa_w^{\tiny{\textcircled{\#}}}a$ and  $fwa_w^{\tiny{\textcircled{\#}}}=fw(aw(a_w^{\tiny{\textcircled{\#}}})^2)=(fwa)w(a_w^{\tiny{\textcircled{\#}}})^2=(wa)w(a_w^{\tiny{\textcircled{\#}}})^2
=w(aw(a_w^{\tiny{\textcircled{\#}}})^2)
=wa_w^{\tiny{\textcircled{\#}}}$. Thus, $awa_w^{\tiny{\textcircled{\#}}}=a(fwa_w^{\tiny{\textcircled{\#}}})=(af)wa_w^{\tiny{\textcircled{\#}}}=(bf)wa_w^{\tiny{\textcircled{\#}}}=b(fwa_w^{\tiny{\textcircled{\#}}})=bwa_w^{\tiny{\textcircled{\#}}}$.
\hfill$\Box$
\end{proof}

Combining with Theorems \ref{projection1} and \ref{idempotent1}, we get several characterizations for the $w$-core partial order $a\overset{\tiny{\textcircled{\#}}}\leq_w b$.

\begin{theorem} \label{char w core} Let $a,b,w\in R$ with $a\in R_w^{\tiny{\textcircled{\#}}}$. Then the following conditions are equivalent{\rm :}

\emph{(i)} $a\overset{\tiny{\textcircled{\#}}}\leq_w b$.

\emph{(ii)} $w^{\parallel a}=w^{\parallel a}a^{(1,3)}b$ and $a=bww^{\parallel a}$.

\emph{(iii)} $a^*a=a^*b$ and $bwa=awa$.

\emph{(iv)} $a^*a=a^*b$ and $b(wa)^\#=a(wa)^\#$.

\emph{(v)} $a=awa_w^{\tiny{\textcircled{\#}}}b=bwa_w^{\tiny{\textcircled{\#}}}a$.

\emph{(vi)}  There exist a projection $p\in R$ and an element $e\in R$ such that $aR=pR$, $pa=pb$, $Re=Raw$, $ea=a$ and $awe=bwe$.

\emph{(vii)}  There exist a projection $p\in R$ and an element $e\in R$ such that ${^0}p={^0}a$, $pa=pb$, $e^0=(aw)^0$, $ea=a$ and $awe=bwe$.

\emph{(viii)}  There exist a projection $p\in R$ and an element $f\in R$ such that $aR=pR$, $pa=pb$, $Ra=Rf$, $af=bf$ and $fwa=wa$.

\emph{(ix)}  There exist a projection $p\in R$ and an element $f\in R$ such that ${^0}p={^0}a$, $pa=pb$, $a^0=f^0$, $af=bf$ and $fwa=wa$.

\emph{(x)} There exist a projection $p\in R$ and an element $f\in R$ such that ${^0}p={^0}a$, $pa=pb$, $a^0\subseteq f^0$, $af=bf$ and $fwa=wa$.

\emph{(xi)} There exist a projection $p\in R$ and an element $e\in R$ such that $pb=a=ea$ and $bwe=aw$.

\emph{(xii)} There exist a projection $p\in R$ and an element $e\in R$ such that $pb=a=ea$ and $bwe=awe$.

\end{theorem}

\begin{proof}

(i)-(x) are equivalent by Theorems \ref{projection1} and \ref{idempotent1}. It next suffices to prove (i) $\Leftrightarrow$ (xi) $\Leftrightarrow$ (xii).

(i) $\Rightarrow$ (xi) Set $p=awa_w^{\tiny{\textcircled{\#}}} $ and $e=a_w^{\tiny{\textcircled{\#}}} aw$, then $p=p^2=p^*$. It is obvious that $ea=a_w^{\tiny{\textcircled{\#}}} awa=a=awa_w^{\tiny{\textcircled{\#}}}a=awa_w^{\tiny{\textcircled{\#}}}b=pb$. Also, $bwe=bwa_w^{\tiny{\textcircled{\#}}} aw=awa_w^{\tiny{\textcircled{\#}}} aw=aw$.

(xi) $\Rightarrow$ (xii) is obvious.

(xii) $\Rightarrow$ (i) It follows from $a=pb$ that $pa=ppb=pb=a$ and hence $a^*b=(pa)^*b=a^*pb=a^*a$. Also, $bwa=bw(ea)=(bwe)a=(awe)a=aw(ea)=awa$. So, $a\overset{\tiny{\textcircled{\#}}}\leq_w b$ by (iii) $\Rightarrow$ (i).
\hfill$\Box$
\end{proof}

Note the fact that $1^{\parallel a}=aa^\#$ provided that $a\in R^\#$. Set $w=1$ in Theorem \ref{char w core}, then the condition (ii) can be reduced to $aa^\#=aa^\#a^{(1,3)}b$ and $a=ba^\#a$, which are equivalent to $a=aa^{(1,3)}b=ba^\#a$. We hence get several characterizations for the core partial order, some of which were given in \cite[Theorems 2.4 and 2.6]{Rakic2015}.

\begin{corollary} \label{char core} Let $a,b\in R$ with $a\in R^{\tiny{\textcircled{\#}}}$. Then the following conditions are equivalent{\rm :}

\emph{(i)} $a\overset{\tiny{\textcircled{\#}}}\leq b$.

\emph{(ii)} $a=aa^{(1,3)}b=ba^\#a$.

\emph{(iii)} $a^*a=a^*b$ and $ba=a^2$.

\emph{(iv)} $a^*a=a^*b$ and $ba^\#=aa^\#$.

\emph{(v)} $a=aa^{\tiny{\textcircled{\#}}}b=ba^{\tiny{\textcircled{\#}}}a$.

\emph{(vi)}  There exist a projection $p\in R$ and an element $e\in R$ such that $aR=pR$, $pa=pb$, $Re=Ra$, $ea=a$ and $ae=be$.

\emph{(vii)}  There exist a projection $p\in R$ and an element $e\in R$ such that ${^0}p={^0}a$, $pa=pb$, $e^0=(a)^0$, $ea=a$ and $ae=be$.

\emph{(viii)}  There exist a projection $p\in R$ and an element $f\in R$ such that $aR=pR$, $pa=pb$, $Ra=Rf$, $af=bf$ and $fa=a$.

\emph{(ix)}  There exist a projection $p\in R$ and an element $f\in R$ such that ${^0}p={^0}a$, $pa=pb$, $a^0=f^0$, $af=bf$ and $fa=a$.

\emph{(x)} There exist a projection $p\in R$ and an element $f\in R$ such that ${^0}p={^0}a$, $pa=pb$, $a^0\subseteq f^0$, $af=bf$ and $fa=a$.

\emph{(xi)} There exist a projection $p\in R$ and an element $e\in R$ such that $pb=a=ea$ and $be=a$.

\emph{(xii)} There exist a projection $p\in R$ and an element $e\in R$ such that $pb=a=ea$ and $be=ae$.
\end{corollary}

\section{Connections with other partial orders}

For any $a,b\in R$, recall that the left star partial order $a ~*\leq b$ is defined as $a^*a=a^*b$ and $aR \subseteq bR$. The right sharp partial order $a\leq_\#b$ is defined by $aa^\#=ba^\#$ and $Ra\subseteq Rb$ for $a\in R^\#$. It is known from \cite{Baksalary2010} that $A\overset{\tiny{\textcircled{\#}}}\leq B$ if and only if $A~*\leq B$ and $A\leq_\#B$, where $A,B$ are $n \times n$ complex matrices of index 1. The characterization for the core partial order indeed holds in a $*$-ring, namely, for any $a,b\in R$ and $a\in R^{\tiny{\textcircled{\#}}}$, then $a\overset{\tiny{\textcircled{\#}}}\leq b$ if and only if $a~*\leq b$ and $a\leq_\#b$.

As stated in Section 1, $a \in R_w^{\tiny{\textcircled{\#}}}$ if and only if $w\in R^{\parallel a}$ and $a\in R^{(1,3)}$. In terms of Lemma \ref{group result}, one knows that $a\in R_w^{\tiny{\textcircled{\#}}}$ implies $wa\in R^\#$. It is natural to ask whether the $w$-core partial order also has a similar characterization, i.e., if $a\in R_w^{\tiny{\textcircled{\#}}}$, whether $a\overset{\tiny{\textcircled{\#}}}\leq_w b$ is equivalent to $a~*\leq b$ and $wa\leq_\#wb$.

The following result gives the implication $a\overset{\tiny{\textcircled{\#}}}\leq_w b \Rightarrow a~*\leq b$ and $wa\leq_\#wb$. For the converse part, there is, of course, a counterexample (see Example \ref{ex2}) to illustrate that it is not true in general.

\begin{proposition} \label{w-core imply} For any $a,b,w\in R$ and $a\in R_w^{\tiny{\textcircled{\#}}}$, if $a\overset{\tiny{\textcircled{\#}}}\leq_w b$, then $a~*\leq b$ and $wa\leq_\#wb$.
\end{proposition}

\begin{proof}
Given $a\overset{\tiny{\textcircled{\#}}}\leq_w b$, i.e., $a_w^{\tiny{\textcircled{\#}}} a=a_w^{\tiny{\textcircled{\#}}} b$ and $awa_w^{\tiny{\textcircled{\#}}} =bwa_w^{\tiny{\textcircled{\#}}}$, then $a=awa_w^{\tiny{\textcircled{\#}}}a=bwa_w^{\tiny{\textcircled{\#}}}a$ and hence $aR\subseteq bR$. It follows from Theorem \ref{char w core} (i) $\Rightarrow$ (iii) that $a^*a=a^*b$. So, $a~*\leq b$.

From $a_w^{\tiny{\textcircled{\#}}} a=a_w^{\tiny{\textcircled{\#}}} b$, we have $R(wa)=R(wawa_w^{\tiny{\textcircled{\#}}}a)=R(wawa_w^{\tiny{\textcircled{\#}}}b)=R(wawa_w^{\tiny{\textcircled{\#}}}(b_w^{\tiny{\textcircled{\#}}}bwb))
\subseteq R(wb)$. Note also that $awa_w^{\tiny{\textcircled{\#}}} =bwa_w^{\tiny{\textcircled{\#}}}$. Then, by Theorem \ref{char w core}, $a(wa)^\#=b(wa)^\#$ and consequently $wa(wa)^\#=wb(wa)^\#$. So, $wa\leq_\#wb$.
\hfill$\Box$
\end{proof}

\begin{example} \label{ex2} {\rm Let $R$ and the involution be the same as that of Example \ref{ex1}. Set $a=\begin{bmatrix}
1 & 1 \\
0 & 0 \\
\end{bmatrix}
$, $w=\begin{bmatrix}
           1 & 0 \\
           0 & 0 \\
         \end{bmatrix}$, $b=\begin{bmatrix}
           1 & 1 \\
           2 & 0 \\
         \end{bmatrix}
\in R$, then $a_w^{\tiny{\textcircled{\#}}}=
\begin{bmatrix}
1 & 0 \\
0 & 0 \\
\end{bmatrix}$. By a direct check,  $a^*a=\begin{bmatrix}
           1 & 0 \\
           1 & 0 \\
         \end{bmatrix}\begin{bmatrix}
           1 & 1 \\
           0 & 0 \\
         \end{bmatrix}=\begin{bmatrix}
           1 & 1 \\
           1 & 1 \\
         \end{bmatrix}=\begin{bmatrix}
           1 & 0 \\
           1 & 0 \\
         \end{bmatrix}\begin{bmatrix}
           1 & 1 \\
           2 & 0 \\
         \end{bmatrix}=a^*b$ and $aR\subseteq bR$. So, $a~*\leq b$. Note that $wa=wb=\begin{bmatrix}
           1 & 1 \\
           0 & 0 \\
         \end{bmatrix}$ are idempotent. Then $Rwa\subseteq Rwb$, $(wa)^\#=(wb)^\#=\begin{bmatrix}
           1 & 1 \\
           0 & 0 \\
         \end{bmatrix}$ and $wa(wa)^\#=\begin{bmatrix}
           1 & 1 \\
           0 & 0 \\
         \end{bmatrix}=wb(wa)^\#$. So, $wa\leq_\#wb$. However, $awa_w^{\tiny{\textcircled{\#}}}=\begin{bmatrix}
           1 & 0 \\
           0 & 0 \\
         \end{bmatrix} \neq \begin{bmatrix}
           1 & 0 \\
           2 & 0 \\
         \end{bmatrix}=bwa_w^{\tiny{\textcircled{\#}}}$. So, $a\overset{\tiny{\textcircled{\#}}}\leq_w b$ does not hold.}
\end{example}

It is of interest to consider, under what conditions, $ a~*\leq b$ and $wa\leq_\#wb$ can imply $a\overset{\tiny{\textcircled{\#}}}\leq_w b$, provided that $a\in R_w^{\tiny{\textcircled{\#}}}$. The result below shows that the implication is true under the hypothesis $w\in U(R)$, where $U(R)$ denotes the group of all units in $R$.

\begin{theorem} \label{w-core left star right sharp} For any $a,b,w\in R$ and $a\in R_w^{\tiny{\textcircled{\#}}}$, if $w\in U(R)$, then $a\overset{\tiny{\textcircled{\#}}}\leq_w b$ if and only if $a~*\leq b$ and $wa\leq_\#wb$.
\end{theorem}

\begin{proof} We only need to prove the ``if'' part. Note that $a~*\leq b$ implies $a^*a=a^*b$ and hence $a_w^{\tiny{\textcircled{\#}}}a=a_w^{\tiny{\textcircled{\#}}}b$ in terms of Theorem \ref{projection1}. To show $a\overset{\tiny{\textcircled{\#}}}\leq_w b$, it suffices to prove $awa=bwa$ by Theorem \ref{char w core}. Since $wa(wa)^\#=wb(wa)^\#$, we have $wa=wb(wa)^\#wa=wbwa(wa)^\#=wbww^{\parallel a}$ and hence $a=w^{-1}wa=w^{-1}wbww^{\parallel a}=bww^{\parallel a}$. Post-multiplying $a=bww^{\parallel a}$ by $wa$ implies $awa=bw(w^{\parallel a}wa)=bwa$, as required.
\hfill$\Box$
\end{proof}

From the proof of Theorem \ref{w-core left star right sharp}, one knows that if $w\in R$ is left invertible, then we also have the equivalence that $a\overset{\tiny{\textcircled{\#}}}\leq_w b$ if and only if $a~*\leq b$ and $wa\leq_\#wb$.

With the next theorem we will present more relationship between the $w$-core partial order and the left star partial order in a $*$-ring $R$.

\begin{theorem} Let $a,b,w\in R$ with $a,b\in R_w^{\tiny{\textcircled{\#}}}$. Then the following conditions are equivalent{\rm :}

\emph{(i)} $a\overset{\tiny{\textcircled{\#}}}\leq_w b$.

\emph{(ii)}  $a~*\leq b$ and $a=bwa_w ^{\tiny{\textcircled{\#}}}b$.

\emph{(iii)}  $a~*\leq b$ and $a_w ^{\tiny{\textcircled{\#}}}=b_w ^{\tiny{\textcircled{\#}}}awa_w ^{\tiny{\textcircled{\#}}}$.

\emph{(iv)}  $a~*\leq b$ and $a_w ^{\tiny{\textcircled{\#}}}=b_w ^{\tiny{\textcircled{\#}}}awb_w ^{\tiny{\textcircled{\#}}}$.
\end{theorem}

\begin{proof}
(i) $\Rightarrow$ (ii) follows from Theorem \ref{char w core} and Proposition \ref{w-core imply}. (i) $\Rightarrow$ (iii) and (i) $\Rightarrow$ (iv) by Propositions \ref{new add pro 2.6} and \ref{w-core imply}.

(ii) $\Rightarrow$ (i) Given $a~*\leq b$, then $a^*a=a^*b$ and hence $a_w ^{\tiny{\textcircled{\#}}}a=a_w ^{\tiny{\textcircled{\#}}}b$ by Theorem \ref{projection1}. Also, one has $awa_w ^{\tiny{\textcircled{\#}}}=(bwa_w ^{\tiny{\textcircled{\#}}}b)wa_w ^{\tiny{\textcircled{\#}}}=bw(a_w ^{\tiny{\textcircled{\#}}}a)wa_w ^{\tiny{\textcircled{\#}}}=bw(a_w ^{\tiny{\textcircled{\#}}}awa_w ^{\tiny{\textcircled{\#}}})=bwa_w ^{\tiny{\textcircled{\#}}}$. So, $a\overset{\tiny{\textcircled{\#}}}\leq_w b$.

(iii) $\Rightarrow$ (i) As $a~*\leq b$, i.e., $a^*a=a^*b$ and $aR\subseteq bR$, then there exists some $x\in R$ such that $a=bx=bwb_w ^{\tiny{\textcircled{\#}}}bx=bwb_w ^{\tiny{\textcircled{\#}}}a$. Hence, $awa_w ^{\tiny{\textcircled{\#}}}=(bwb_w ^{\tiny{\textcircled{\#}}}a)wa_w ^{\tiny{\textcircled{\#}}}=bw(b_w ^{\tiny{\textcircled{\#}}}awa_w ^{\tiny{\textcircled{\#}}})=bwa_w ^{\tiny{\textcircled{\#}}}$. Since $a^*a=a^*b$, $a_w ^{\tiny{\textcircled{\#}}}a=a_w ^{\tiny{\textcircled{\#}}}b$ by Theorem \ref{projection1}.

(iv) $\Rightarrow$ (i) By (ii) $\Rightarrow$ (i), $a_w ^{\tiny{\textcircled{\#}}}a=a_w ^{\tiny{\textcircled{\#}}}b$. Next, it is only need to show $awa_w ^{\tiny{\textcircled{\#}}}=bwa_w ^{\tiny{\textcircled{\#}}}$. Note that $aR \subseteq bR$ implies $a=bwb_w ^{\tiny{\textcircled{\#}}}a$. So, $bwa_w ^{\tiny{\textcircled{\#}}}=bw(b_w ^{\tiny{\textcircled{\#}}}awb_w ^{\tiny{\textcircled{\#}}})=(bwb_w ^{\tiny{\textcircled{\#}}}a)wb_w ^{\tiny{\textcircled{\#}}}=awb_w ^{\tiny{\textcircled{\#}}}$.

Hence, we have
\begin{eqnarray*}
awa_w ^{\tiny{\textcircled{\#}}}&=&aw(b_w ^{\tiny{\textcircled{\#}}}awb_w ^{\tiny{\textcircled{\#}}})=awb_w ^{\tiny{\textcircled{\#}}}aw(b_w ^{\tiny{\textcircled{\#}}}bwb_w ^{\tiny{\textcircled{\#}}})\\
&=&aw(b_w ^{\tiny{\textcircled{\#}}}awb_w ^{\tiny{\textcircled{\#}}})bwb_w ^{\tiny{\textcircled{\#}}}=awa_w ^{\tiny{\textcircled{\#}}}bwb_w ^{\tiny{\textcircled{\#}}}\\
&=&aw(a_w ^{\tiny{\textcircled{\#}}}a)wb_w ^{\tiny{\textcircled{\#}}}=awb_w ^{\tiny{\textcircled{\#}}}\\
&=&bwa_w ^{\tiny{\textcircled{\#}}}.
\end{eqnarray*}

The proof is completed. \hfill$\Box$
\end{proof}

Similarly, the connection between the $w$-core partial order and the right sharp partial order can be given. The proof is left to the reader.

\begin{theorem} Let $a,b,w\in R$ with $a,b\in R_w^{\tiny{\textcircled{\#}}}$. If $w\in U(R)$, then the following conditions are equivalent{\rm :}

\emph{(i)} $a\overset{\tiny{\textcircled{\#}}}\leq_w b$.

\emph{(ii)} $wa\leq_\# wb$ and $a=bwa_w ^{\tiny{\textcircled{\#}}}b$.

\emph{(iii)} $wa\leq_\# wb$ and $a_w ^{\tiny{\textcircled{\#}}}=a_w ^{\tiny{\textcircled{\#}}}awb_w ^{\tiny{\textcircled{\#}}}$.
\end{theorem}

As is noted in Example \ref{ex1}, the $w$-core partial order may not imply the core partial order. We next consider under what conditions the $w$-core partial order gives a core partial order. Herein, a lemma is presented, which was indeed given in \cite[Theorem 2.25]{Zhu2020}. For completeness, we give its proof.

\begin{lemma} \label{Zhulemma} Let $a,w\in R$. Then $a\in R_w^{\tiny\textcircled{\tiny{\#}}}$ if and only if $aR= awR$ and $aw\in R^{\tiny\textcircled{\tiny{\#}}}$. In this case, $a_w^{\tiny\textcircled{\tiny{\#}}}=(aw)^{\tiny\textcircled{\tiny{\#}}}$.
\end{lemma}

\begin{proof}

Suppose $a\in R_w^{\tiny\textcircled{\tiny{\#}}}$. Then $w\in R^{\parallel a}$ and hence $a\in awaR\subseteq awR$. Also, there exists some $x\in R$ such that $xawa=a$, $awx^2=x$ and $(awx)^*=awx$, which guarantee $xawaw=aw$, $awx^2=x$ and $(awx)^*=awx$. So $aw\in R^{\tiny\textcircled{\tiny{\#}}}$.

Conversely, as $aw\in R^{\tiny\textcircled{\tiny{\#}}}$, then there exists some $y\in R$ such that $awyaw=aw$, $yawy=y$, $y(aw)^2=aw$, $awy^2=y$ and $awy=(awy)^*$. Since $aR=awR$, $a=awt$ for some $t\in R$ and hence $a=awt=(awyaw)t=awya=(awy)^*a=(wy)^*a^*a\in Ra^*a$, i.e., $a\in R^{(1,3)}$.

Note the fact that $(aw)^{\tiny\textcircled{\tiny{\#}}}=(aw)^\#aw(aw)^{(1,3)}=w^{\parallel a}w(aw)^{(1,3)}$. To show that $w^{\parallel a}w(aw)^{(1,3)}$ is the $w$-core inverse of $a$, it suffices to prove that $z=w(aw)^{(1,3)}$ is a $\{1,3\}$-inverse of $a$. Since $a\in awR$, it follows that $a=awt$ for some $t\in R$ and $a=aw(aw)^{(1,3)}awt=aw(aw)^{(1,3)}a=aza$. Also, $az=aw(aw)^{(1,3)}=(az)^*$, as required. \hfill$\Box$
\end{proof}

\begin{theorem} \label{w-core and core} For any $a,b,w\in R$ and $a\in R_w^{\tiny{\textcircled{\#}}}$, if $w\in U(R)$, then the following conditions are equivalent{\rm :}

\emph{(i)} $a\overset{\tiny{\textcircled{\#}}}\leq_w b$.

\emph{(ii)} $aw\overset{\tiny{\textcircled{\#}}}\leq bw$.
\end{theorem}

\begin{proof}
(i) $\Rightarrow$ (ii) Suppose $a\overset{\tiny{\textcircled{\#}}}\leq_w b$, i.e., $a_w^{\tiny{\textcircled{\#}}} a=a_w^{\tiny{\textcircled{\#}}} b$ and $awa_w^{\tiny{\textcircled{\#}}} =bwa_w^{\tiny{\textcircled{\#}}}$. As $a\in R_w^{\tiny{\textcircled{\#}}}$, then $aw\in R^{\tiny{\textcircled{\#}}}$ and $a_w^{\tiny\textcircled{\tiny{\#}}}=(aw)^{\tiny\textcircled{\tiny{\#}}}$ by Lemma \ref{Zhulemma}. So, $aw(aw)^{\tiny{\textcircled{\#}}} =bw(aw)^{\tiny{\textcircled{\#}}}$, $(aw)^{\tiny{\textcircled{\#}}} a=(aw)^{\tiny{\textcircled{\#}}} b$ and hence $(aw)^{\tiny{\textcircled{\#}}} aw=(aw)^{\tiny{\textcircled{\#}}} bw$. So, $aw\overset{\tiny{\textcircled{\#}}}\leq bw$.

(ii) $\Rightarrow$ (i) As $aw\overset{\tiny{\textcircled{\#}}}\leq bw$, then $(aw)^{\tiny{\textcircled{\#}}} aw=(aw)^{\tiny{\textcircled{\#}}} bw$. Pre-multiplying $(aw)^{\tiny{\textcircled{\#}}} aw=(aw)^{\tiny{\textcircled{\#}}} bw$ by $aw$ gives $aw=aw(aw)^{\tiny{\textcircled{\#}}} bw$. Post-multiplying $aw=aw(aw)^{\tiny{\textcircled{\#}}} bw$ by $w^{\parallel a}$ implies $a=aww^{\parallel a}=aw(aw)^{\tiny{\textcircled{\#}}} bww^{\parallel a}\in awR$. So, $a_w^{\tiny\textcircled{\tiny{\#}}}=(aw)^{\tiny\textcircled{\tiny{\#}}}$ by Lemma \ref{Zhulemma}. Thus, $aw(aw)^{\tiny\textcircled{\tiny{\#}}}=bw(aw)^{\tiny\textcircled{\tiny{\#}}}$ guarantees that $awa_w^{\tiny\textcircled{\tiny{\#}}}=bwa_w^{\tiny\textcircled{\tiny{\#}}}$, and $(aw)^{\tiny{\textcircled{\#}}} aw=(aw)^{\tiny{\textcircled{\#}}} bw$ gives $a_w^{\tiny{\textcircled{\#}}} aw=a_w^{\tiny{\textcircled{\#}}} bw$. Since $w\in U(R)$, $a_w^{\tiny{\textcircled{\#}}} a=a_w^{\tiny{\textcircled{\#}}} b$, as required.
\hfill$\Box$
\end{proof}

Combining with Theorems \ref{w-core left star right sharp} and \ref{w-core and core}, we have the following result.

\begin{theorem} For any $a,b,w\in R$ and $a\in R_w^{\tiny{\textcircled{\#}}}$, if $w\in U(R)$, then the following conditions are equivalent{\rm :}

\emph{(i)} $a\overset{\tiny{\textcircled{\#}}}\leq_w b$.

\emph{(ii)} $aw\overset{\tiny{\textcircled{\#}}}\leq bw$.

\emph{(iii)} $a~*\leq b$ and $wa\leq_\#wb$.
\end{theorem}

We next show that star partial order $\overset{*}\leq$ and the core partial order $\overset{\tiny{\textcircled{\#}}}\leq$ are instances of the $w$-core partial order $\overset{\tiny{\textcircled{\#}}}\leq_w$.

\begin{theorem} \label{three class partial orders} Let $a,b\in R$. Then we have

\emph{(i)} If $a\in R^{\tiny{\textcircled{\#}}}$, then $a\overset{\tiny{\textcircled{\#}}}\leq b$ if and only if $a\overset{\tiny{\textcircled{\#}}}\leq_a b$ if and only if $a\overset{\tiny{\textcircled{\#}}}\leq_1 b$.

\emph{(ii)} If $a\in R^\dag$, then $a\overset{*}\leq b$ if and only if $a\overset{\tiny{\textcircled{\#}}}\leq_{a^*} b$.
\end{theorem}

\begin{proof}

(i)  It is known that  $a\in R_a^{\tiny{\textcircled{\#}}}$ if and only if  $a\in R^{\tiny{\textcircled{\#}}}$ if and only if $a\in R_1^{\tiny{\textcircled{\#}}}$. We first show that $a\overset{\tiny{\textcircled{\#}}}\leq b$ if and only if $a\overset{\tiny{\textcircled{\#}}}\leq_a b$.

Suppose $a\overset{\tiny{\textcircled{\#}}}\leq b$, i.e., $a^{\tiny{\textcircled{\#}}}a=a^{\tiny{\textcircled{\#}}}b$ and $aa^{\tiny{\textcircled{\#}}}=ba^{\tiny{\textcircled{\#}}}$. Then $a^\#a=a^\#aa^{(1,3)}b$ and $aa^{(1,3)}=ba^\#aa^{(1,3)}$, and consequently $a_a^{\tiny{\textcircled{\#}}}a=a^\#a^{(1,3)}a=(a^\#)^2aa^{(1,3)}a=a^\#a^\#a=a^\#a^\#aa^{(1,3)}b=a^\#a^{(1,3)}b=a_a^{\tiny{\textcircled{\#}}}b$. Similarly, $a^2a_a^{\tiny{\textcircled{\#}}}=a^2a^\#a^{(1,3)}=aa^{(1,3)}=ba^\#aa^{(1,3)}=baa^\#a^{(1,3)}=baa_a^{\tiny{\textcircled{\#}}}$.

Conversely, if $a\overset{\tiny{\textcircled{\#}}}\leq_a b$, then $a_a^{\tiny{\textcircled{\#}}}a=a_a^{\tiny{\textcircled{\#}}}b$ and $a^2a_a^{\tiny{\textcircled{\#}}}=baa_a^{\tiny{\textcircled{\#}}}$, that is, $a^\#a^{(1,3)}a=a^\#a^{(1,3)}b$ and $a^2a^\#a^{(1,3)}=baa^\#a^{(1,3)}$. So, $a^{\tiny{\textcircled{\#}}}a=a^\#aa^{(1,3)}a=aa^\#a^{(1,3)}a=aa^\#a^{(1,3)}b=a^{\tiny{\textcircled{\#}}}b$ and $aa^{\tiny{\textcircled{\#}}}=aa^{(1,3)}=a^2a^\#a^{(1,3)}=baa^\#a^{(1,3)}=ba^{\tiny{\textcircled{\#}}}$. So, $a\overset{\tiny{\textcircled{\#}}}\leq b$.

Note that $a^{\tiny{\textcircled{\#}}}=a_1^{\tiny{\textcircled{\#}}}$. Then $a\overset{\tiny{\textcircled{\#}}}\leq b$ if and only if $a\overset{\tiny{\textcircled{\#}}}\leq_1 b$. So, the result follows.

(ii) One knows that $a\in R^\dag$ if and only if $(a^*)^{\parallel a}$ exists if and only if $a\in R_{a^*}^{\tiny{\textcircled{\#}}}$. Moreover, $a_{a^*}^{\tiny{\textcircled{\#}}}=(a^*)^{\parallel a}a^{(1,3)}=(a^\dag)^*a^{(1,3)}=(a^\dag)^*a^\dag$. It is also known that, for any $a\in R^\dag$, $a\overset{*}\leq b$ if and only if $a^\dag a=a^\dag b$ and $aa^\dag=ba^\dag$. We next only show that $a\overset{\tiny{\textcircled{\#}}}\leq_{a^*} b$ if and only if $a^\dag a=a^\dag b$ and $aa^\dag=ba^\dag$.

Suppose $a^\dag a=a^\dag b$ and $aa^\dag=ba^\dag$. We hence have $a_{a^*}^{\tiny{\textcircled{\#}}}a=(a^\dag)^*a^\dag a=(a^\dag)^*a^\dag b=a_{a^*}^{\tiny{\textcircled{\#}}}b$, and $aa^*a_{a^*}^{\tiny{\textcircled{\#}}}=aa^*(a^\dag)^*a^\dag=a(a^\dag a)^*a^\dag=aa^\dag=ba^\dag=ba^*(a^\dag)^*a^\dag=ba^*a_{a^*}^{\tiny{\textcircled{\#}}}$.

For the converse part, given $a_{a^*}^{\tiny{\textcircled{\#}}}a=a_{a^*}^{\tiny{\textcircled{\#}}}b$ and $aa^*a_{a^*}^{\tiny{\textcircled{\#}}}=ba^*a_{a^*}^{\tiny{\textcircled{\#}}}$, i.e., $(a^\dag)^*a^\dag a=(a^\dag)^*a^\dag b$ and $aa^*(a^\dag)^*a^\dag=ba^*(a^\dag)^*a^\dag$, then $a^\dag a=a^\dag aa^\dag a=(a^\dag a)^*a^\dag a=a^*(a^\dag)^*a^\dag a=a^*(a^\dag)^*a^\dag b=(a^\dag a)^*a^\dag b=a^\dag b$. Similarly, $aa^\dag=a(a^\dag a)^*a^\dag=aa^*(a^\dag)^*a^\dag=ba^*(a^\dag)^*a^\dag=b(a^\dag a)^*a^\dag=ba^\dag$, as required.
\hfill$\Box$
\end{proof}

An element $a\in R$ is called EP if $a\in R^\#\cap R^\dag$ and $a^\#=a^\dag$. A well known characterization for EP elements is that $a$ is EP if and only if $a\in R^\dag$ and $aa^\dag=a^\dag a$. According to Theorem \ref{three class partial orders}, we have the following result, of which (iii) $\Leftrightarrow$ (iv) $\Leftrightarrow$ (v) were essentially given in \cite{Baksalary2010}.

\begin{theorem}  Let $a,b\in R$. If $a$ is EP, then the following conditions are equivalent{\rm :}

\emph{(i)} $a\overset{\tiny{\textcircled{\#}}}\leq_a b$.

\emph{(ii)} $a\overset{\tiny{\textcircled{\#}}}\leq_{a^*} b$.

\emph{(iii)} $a\overset{\tiny{\textcircled{\#}}}\leq b$.

\emph{(iv)} $a\overset{*}\leq b$.

\emph{(v)} $a\overset{\#}\leq b$.
\end{theorem}

Given any $a,b\in R$ with $a\in R_w^{\tiny{\textcircled{\#}}}$, the $w$-core partial order gives the diamond partial order, i.e., $a\overset{\tiny{\textcircled{\#}}}\leq_w b \Rightarrow a\overset{\diamond}\leq b$. Indeed, given $a\overset{\tiny{\textcircled{\#}}}\leq_w b$, by Theorem \ref{char w core}, we have $a^*a=a^*b=b^*a$ and hence $aa^*a=ab^*a$. Moreover, $a=bwa_w^{\tiny\textcircled{\tiny{\#}}}a=awa_w^{\tiny\textcircled{\tiny{\#}}}b$ imply $aR \subseteq bR$ and $Ra\subseteq Rb$, respectively.

For any $a,b,w\in R$ and $a\in R^{\tiny{\textcircled{\#}}} \cap R_w^{\tiny{\textcircled{\#}}}$, we claim that the $w$-core partial order is between the core partial order and the diamond partial order, that is, $a\overset{\tiny{\textcircled{\#}}}\leq b \Rightarrow a\overset{\tiny{\textcircled{\#}}}\leq_w b \Rightarrow a\overset{\diamond}\leq b$. However, the converse implications may not be true in general, i.e., $ a\overset{\diamond}\leq b \nRightarrow a\overset{\tiny{\textcircled{\#}}}\leq_w b \nRightarrow a\overset{\tiny{\textcircled{\#}}}\leq b$ (see Examples \ref{ex1} and \ref{ex2}).

The equivalences between $a\overset{x}\leq b$ and $b-a\overset{x}\leq b$ were considered by several scholars \cite{Ferreyra2020,Marovt2016,Mitra1987}, where $\overset{x}\leq$ denotes the minus partial order $\overset{-}\leq$, the star partial order $\overset{*}\leq$ or the sharp partial order $\overset{\#}\leq$. The characterization fails to hold for the core partial order $\overset{\tiny{\textcircled{\#}}}\leq$ (see, e.g., \cite[p. 695]{Baksalary2010}). Recently, Ferreyra and Malik \cite[Theorem 4.2]{Ferreyra2020} derived the equivalence between  $A\overset{\tiny{\textcircled{\#}}}\leq B$ and $B-A\overset{\tiny{\textcircled{\#}}}\leq B$ in the ring of all $n\times n$ complex matrices, under certain conditions. More precisely, if $A$, $B$ and $B-A$ are group invertible complex matrices of $n$ by $n$ size, then $A\overset{\tiny{\textcircled{\#}}}\leq B$ and $AB=BA$ if and only if $B-A\overset{\tiny{\textcircled{\#}}}\leq B$ and $AB=BA$ if and only if $A~*\leq B$ and $A\overset{\#}\leq B$.

We next give a similar characterization for the $w$-core inverse by a pure algebraic method in a $*$-ring.

\begin{theorem} \label{difference w-core} For any $a,b,w\in R$ with $awb=bwa$, let $a,b\in R_w^{\tiny{\textcircled{\#}}}$ such that $a-b\in R_w^{\tiny{\textcircled{\#}}}$. Then the following conditions are equivalent{\rm:}

\emph{(i)} $a\overset{\tiny{\textcircled{\#}}}\leq_w b$.

\emph{(ii)} $b-a\overset{\tiny{\textcircled{\#}}}\leq_w b$.

\emph{(iii)} $a~*\leq b$ and $wa\overset{\#}\leq wb$.
\end{theorem}

\begin{proof}
(i) $\Rightarrow$ (ii) Suppose $a\overset{\tiny{\textcircled{\#}}}\leq_w b$. Then, by Theorem \ref{char w core}, $a^*a=a^*b$ and $bwa=awa$. To show $b-a\overset{\tiny{\textcircled{\#}}}\leq_w b$, it suffices to prove $(b-a)^*(b-a)=(b-a)^*b$ and $(b-a)w(b-a)=bw(b-a)$. As $(b-a)^*a=0$, then $(b-a)^*(b-a)=(b-a)^*b-(b-a)^*a=(b-a)^*b$. Similarly, we get $(b-a)w(b-a)=bw(b-a)-aw(b-a)=bw(b-a)$.

(ii) $\Rightarrow$ (iii) Given $b-a\overset{\tiny{\textcircled{\#}}}\leq_w b$, then $(b-a)^*(b-a)=(b-a)^*b$, which implies $(b-a)^*a=0$ and so $a^*a=a^*b$. Also, $(b-a)w(b-a)=bw(b-a)$ gives $awb=awa$ and $bwa=awa$. Post-multiplying $bwa=awa$ by $w(a_w^{\tiny{\textcircled{\#}}})^2$ gives $bwa_w^{\tiny{\textcircled{\#}}}=awa_w^{\tiny{\textcircled{\#}}}$. Hence, $a=awa_w^{\tiny{\textcircled{\#}}}a=bwa_w^{\tiny{\textcircled{\#}}}a\in bR$ and $a~*\leq b$. Again, it follows from $awa=awb=bwa$ that $wawa=wawb=wbwa$. We hence have $(wa)^\#wa=(wa)^\#wb$ and $wa(wa)^\#=wb(wa)^\#$. So, $wa\overset{\#}\leq wb$.

(iii) $\Rightarrow$ (i) Note that $a~*\leq b$ implies $a^*a=a^*b$. To prove $a\overset{\tiny{\textcircled{\#}}}\leq_w b$, we only need to verify that $bwa=awa$ by Theorem \ref{char w core}. Since $wa\overset{\#}\leq wb$, $(wa)^\#wa=(wa)^\#wb$ and $(wa)^2=wawb$. Pre-multiplying $(wa)^2=wawb$ by $w^{\parallel a}$ implies $awa=(w^{\parallel a}wa)wa=(w^{\parallel a}wa) wb=awb$, which together with $awb=bwa$ to guarantee $bwa=awa$, as required.
\hfill$\Box$
\end{proof}

\begin{remark} \label{rmk2} {\rm It should be noted that, in Theorem \ref{difference w-core} above, the condition (iii) cannot imply the condition (i) without the assumption $awb=bwa$ in general. Indeed, Example \ref{ex2} can illustrate this fact. However, for the case of $w=1$, the implication (iii) $\Rightarrow$ (i) is clear by the fact that $a\overset{\#}\leq b$ gives $a^2=ab=ba$. }
\end{remark}

As is given in Theorem \ref{three class partial orders} above, for any $a\in R$, $a\in R_a^{\tiny{\textcircled{\#}}}$ if and only if  $a\in R^{\tiny{\textcircled{\#}}}$ if and only if  $a\in R_1^{\tiny{\textcircled{\#}}}$. In terms of Theorem \ref{three class partial orders} and Remark \ref{rmk2}, we get the characterization for the core partial order $a\overset{\tiny{\textcircled{\#}}}\leq b$.

\begin{corollary} For any $a,b\in R$, let $a,b\in R^{\tiny{\textcircled{\#}}}$ such that $a-b\in R^{\tiny{\textcircled{\#}}}$. Then the following conditions are equivalent{\rm:}

\emph{(i)} $a\overset{\tiny{\textcircled{\#}}}\leq b$ and $ab=ba$.

\emph{(ii)} $b-a\overset{\tiny{\textcircled{\#}}}\leq b$  and $ab=ba$.

\emph{(iii)} $a~*\leq b$ and $a\overset{\#}\leq b$.
\end{corollary}

Set $w=a$ in Theorem \ref{difference w-core}, another characterization for the core partial order $a\overset{\tiny{\textcircled{\#}}}\leq b$ can be obtained as follows.

\begin{corollary} For any $a,b\in R$ with $a^2b=ba^2$, let $a,b\in R^{\tiny{\textcircled{\#}}}$ such that $a-b\in R^{\tiny{\textcircled{\#}}}$. Then the following conditions are equivalent{\rm:}

\emph{(i)} $a\overset{\tiny{\textcircled{\#}}}\leq b$.

\emph{(ii)} $b-a\overset{\tiny{\textcircled{\#}}}\leq b$.

\emph{(iii)} $a~*\leq b$ and $a^2\overset{\#}\leq ab$.

\end{corollary}

For any $a,b\in U(R)$, it is well known that $(ab)^{-1}=b^{-1}a^{-1}$, the formula above is well known as the reverse order law. Reverse order laws for the group inverse, the Moore-Penrose inverse and the core inverse do not hold in general. For the case of the the reverse order for the core inverse, a counterexample was constructed in \cite{Cohen2012} to show that $(ab)^{\tiny{\textcircled{\#}}}= b^{\tiny{\textcircled{\#}}}a^{\tiny{\textcircled{\#}}}$ does not hold. Later, Malik et al. \cite{Malik2014} showed the reverse order law for the core inverse of $AB$, under the core partial order $A\overset{\tiny{\textcircled{\#}}}\leq B$, where $A$ and $B$ are two $n \times n$ complex matrices.

A natural question is that whether the $w$-core inverse shares the reverse order law property under the $w$-core partial order, i.e., whether $(ab)_w^{\tiny{\textcircled{\#}}}= b_w^{\tiny{\textcircled{\#}}}a_w^{\tiny{\textcircled{\#}}}$ under the $w$-core partial order $a\overset{\tiny{\textcircled{\#}}}\leq_w b$. Example \ref{counterexample reverse} below illustrates that the hypothesis is not accurate in general.

\begin{example} \label{counterexample reverse} {\rm Let $R$ and the involution be the same as that of Example \ref{ex1}. Take
$a=b=\begin{bmatrix}
       1 & 1 \\
       0 & 0 \\
     \end{bmatrix}
$, $w=
\begin{bmatrix}
       1 & 0 \\
       1& 0 \\
     \end{bmatrix}\in R$, then $(ab)_w^{\tiny{\textcircled{\#}}}=a_w^{\tiny{\textcircled{\#}}}=b_w^{\tiny{\textcircled{\#}}}=
     \begin{bmatrix}
       \frac{1}{2} & 0 \\
       0& 0 \\
     \end{bmatrix}$ by Example \ref{ex1}. By a direct check, $a_w^{\tiny{\textcircled{\#}}}a=a_w^{\tiny{\textcircled{\#}}}b=
     \begin{bmatrix}
       \frac{1}{2} & \frac{1}{2} \\
       0& 0 \\
     \end{bmatrix}$ and $awa_w^{\tiny{\textcircled{\#}}}=bwa_w^{\tiny{\textcircled{\#}}}=\begin{bmatrix}
       1 & 0 \\
       0& 0 \\
     \end{bmatrix}$, i.e., $a\overset{\tiny{\textcircled{\#}}}\leq_w b$. However,
     $\begin{bmatrix}
       \frac{1}{2} & 0 \\
       0& 0 \\
     \end{bmatrix}=(ab)_w^{\tiny{\textcircled{\#}}} \neq b_w^{\tiny{\textcircled{\#}}} a_w^{\tiny{\textcircled{\#}}}=
     \begin{bmatrix}
       \frac{1}{2} & 0 \\
       0& 0 \\
     \end{bmatrix}\begin{bmatrix}
       \frac{1}{2} & 0 \\
       0& 0 \\
     \end{bmatrix}=
     \begin{bmatrix}
       \frac{1}{4} & 0 \\
       0& 0 \\
     \end{bmatrix}$.}
\end{example}

It is natural to ask what types of the product has the reverse order law property, under the assumption $a\overset{\tiny{\textcircled{\#}}}\leq_w b$. The following result illustrates that the $w$-core inverse of $awb$ has the reverse order law property.

\begin{theorem} Let $a,b,w\in R$ with $a,b\in R_w^{\tiny{\textcircled{\#}}}$. If $a\overset{\tiny{\textcircled{\#}}}\leq_w b$, then $awb\in R_w^{\tiny{\textcircled{\#}}}$ and $(awb)_w^{\tiny{\textcircled{\#}}}= b_w^{\tiny{\textcircled{\#}}}a_w^{\tiny{\textcircled{\#}}}$.
\end{theorem}

\begin{proof}
It follows from Lemma \ref{w-core par Lemma} that $b_w^{\tiny{\textcircled{\#}}}a_w^{\tiny{\textcircled{\#}}}=(a_w^{\tiny{\textcircled{\#}}})^2$. We next show that $x=(a_w^{\tiny{\textcircled{\#}}})^2$ is the $w$-core inverse of $awb$ by the following three steps.

(1) Note that $bwa=awa$ and $a_w^{\tiny{\textcircled{\#}}}awa=a$. Then $x(awb)w(awb)=(a_w^{\tiny{\textcircled{\#}}})^2aw(bwa)wb=(a_w^{\tiny{\textcircled{\#}}})^2aw(awa)wb
=a_w^{\tiny{\textcircled{\#}}}(a_w^{\tiny{\textcircled{\#}}}awa)wawb=(a_w^{\tiny{\textcircled{\#}}}awa)wb=awb$.

(2) $(awb)wx=awbw(a_w^{\tiny{\textcircled{\#}}})^2=awaw(a_w^{\tiny{\textcircled{\#}}})^2=awa_w^{\tiny{\textcircled{\#}}}=((awb)wx)^*$.

(3) $(awb)wx^2=(awbwx)x=awa_w^{\tiny{\textcircled{\#}}}(a_w^{\tiny{\textcircled{\#}}})^2=aw(a_w^{\tiny{\textcircled{\#}}})^2 a_w^{\tiny{\textcircled{\#}}}=x$.
\hfill$\Box$
\end{proof}

\bigskip
\centerline {\bf ACKNOWLEDGMENTS}
\vskip 2mm

The authors are highly grateful to the referees for their valuable comments and suggestions
which greatly improved this paper. This research is supported by the National Natural Science Foundation of China (No. 11801124) and China Postdoctoral Science Foundation (No. 2020M671068).

\begin{flushleft}
{\bf References}
\end{flushleft}

\end{document}